\documentclass{amsart}
\RequirePackage{float}%图片浮动
\RequirePackage{graphicx}% 插图包
\RequirePackage{amsmath, amsthm, amssymb, amsfonts}%数学环境
\RequirePackage[colorlinks=true]{hyperref}%交叉引用
\hypersetup{urlcolor=blue, citecolor=blue}%链接颜色
\RequirePackage[nameinlink]{cleveref}%智能引用
\RequirePackage{indentfirst}
\RequirePackage{color}

\numberwithin{equation}{section}
\newtheorem{thm}{Theorem}[section]
\newtheorem{prop}[thm]{Proposition}
\newtheorem{coro}[thm]{Corollary}
\newtheorem{lemma}[thm]{Lemma}

\theoremstyle{remark}
\newtheorem{remark}{Remark}[section] 
\theoremstyle{definition}
\newtheorem{defi}{Definition}%[section]
\crefname{equation}{}{}%配套于cleverref 
\newcommand{\Th}[1]{Theorem \ref{#1}} 
 
\newcommand{\Le}[1]{Lemma \ref{#1}}

\newcommand{\De}[1]{Definition \ref{#1}}
\renewcommand{\Pr}[1]{Proposition \ref{#1}}

\newenvironment{Eq}{\begin{equation}\begin{aligned}}{\end{aligned}\end{equation}\ignorespacesafterend}
\newenvironment{Eq*}{\begin{equation*}\begin{aligned}}{\end{aligned}\end{equation*}\ignorespacesafterend}
\newcounter{head}

\DeclareMathOperator{\supp}{supp}
\DeclareMathOperator{\fd}{d}	\renewcommand{\d}{\fd\!}

\newcommand{\R}{{\mathbb{R}}}

\renewcommand{\P}[2]{\ifthenelse{\equal{#1}{}}{\partial_{#2}}{\frac{\partial{#1}}{\partial {#2}}}}

\newcommand{\kl}[1]{\mathopen{}\left#1}
\newcommand{\kr}[1]{\right#1}

\newcommand{\Roma}[1]{\uppercase\expandafter{\romannumeral#1}}

\makeatletter
\newcommand\xrarrow[2]{\ext@arrow 0359{\arrowfill@{-}{-}{\rightarrow}}{#2}{#1}}
\newcommand\xlarrow[2]{\ext@arrow 6095{\arrowfill@{\leftarrow}{-}{-}}{#2}{#1}}
\makeatother

\title[wave equation with damping and potential]
{Blow-up for Strauss type wave equation with damping and potential}

\author{Wei Dai}
\address{W. Dai \newline
School of Mathematical Sciences,
Zhejiang University,
Hangzhou 310027, P.R.China}
\email{daiw16@zju.edu.cn}
\author{Hideo Kubo}
\address{H. Kubo \newline
Department of Mathematics,
Faculty of Science, Hokkaido University,
Sapporo 060-0810, Japan}
\email{kubo@math.sci.hokudai.ac.jp}
\author{Motohiro Sobajima}
\address{M. Sobajima \newline
Department of Mathematics, 
Faculty of Science and Technology, Tokyo University of Science,  
2641 Yamazaki, Noda-shi, Chiba, 278-8510, Japann}
\email{msobajima1984@gmail.com}
\date{\today}
\begin{document}

\bibliographystyle{elsarticle-num-names}%plain}

\begin{abstract}
We study a kind of nonlinear wave equations with damping and potential, whose coefficients are both critical in the sense of the scaling and depend only on the spatial variables. Based on the earlier works, one may think there are two kinds of blow-up phenomenons when the exponent of the nonlinear term is small. It also means there are two kinds of law to determine the critical exponent. In this paper, we obtain a blow-up result and get the estimate of the upper bound of the lifespan in critical and sub-critical cases. All of the results support such a conjecture, although for now, the existence part is still open. 
\end{abstract}



\maketitle

%\tableofcontents
%\thanks{$^*$ Corresponding author: Wei Dai, E-mail address: daiw16@zju.edu.cn
\section{Introduction}
In this paper, 
we consider the blow-up phenomenon for the following initial value problem of the semilinear wave equation with damping term and potential term:
\begin{Eq}\label{eq:u}
\begin{cases}
Pu:=\partial_t^2u + Ar^{-1}\partial_tu -\Delta u +Br^{-2}u=|u|^p,
\quad t>0,r:=|x|,\\
u(0,x)=\varepsilon f(x), \quad \partial_t u(0,x)=\varepsilon g(x),\quad x\in\R^n
\end{cases}
\end{Eq}
where $n\geq 1$, 
$A\geq 0$, 
$B\in \R$,
$p>1$ and $1\gg\varepsilon>0$. 

%To consider the blow-up phenomenon of its solution, 
%we assume $f$ and $g$ satisfy some conditions on positivity with compact supports.
%It is easy to find out that when $p$ is big enough 
%and $\varepsilon$ is small enough,
%such equation always admits a global solution. 
%Otherwise, 
%when $p$ is small, 
%the solution will blow-up in finite time no matter how small $\varepsilon$ is. 
%So, 
%we expect that there exists a critical exponent $p_c(n,A,B)$,
%which separates global small data solutions 
%and blow-up of the small data solution in finite time.

When $A=B=0$,
the study of \eqref{eq:u} has gone through a long history (see for instance, \cite{KubOht05, MR3013062}),
and it is well known that in this situation there exists a critical exponent
which separates global in time behavior of small amplitude solutions
({\it i.e.}, global existence and blow-up).
The critical exponent is given by the \emph{Strauss} exponent, that is, the positive root of the quadratic equation:
\begin{Eq}\label{eq:hs}
h_S(n,p):=(n-1)p^2-(n+1)p-2=0.
\end{Eq}
This type of results was firstly appeared in \cite{MR535704},
and finally proved by \cite{MR744303} for the subcritical case,
\cite{MR2195336} for the critical case 
and \cite{MR1481816} for the supercritical case.

For \cref{eq:u} only with the damping term,
which means $A>0$, $B=0$,
one related work comes from \cite{ikeda1709life}.
Roughly speaking, 
the authors proved a blow-up result when $p$ is smaller or equal to the critical exponent of the case
$A=B=0$ with $n$ replaced by $n+A$.

We also mention that when $B=0$ and the coefficient of the damping term is given by $(1+|x|^2)^{-\alpha/2}$ with $0\leq \alpha<1$, the critical exponent becomes to $1+2/(n-\alpha)$ which corresponds to the \emph{Fujita} exponent. Such results have appeared in \cite{MR2589664} and some other works.

In our paper, we will expose both of these blow-up phenomena corresponding to the so-called \emph{Strauss} and \emph{Fujita} exponents and get the upper bound of the lifespan for \cref{eq:u}.
Though less is known about the existence results for such an equation, 
we believe our bounds are sharp in general. 
This is because our bounds coincide with the lower bounds for almost all of the already known cases, for instance \cite{MR1408499}, \cite{MR1233659}, \cite{MR3953034} and so on.

Another feature of the equation in \eqref{eq:u} shows up from the presence of the inverse square potential $Br^{-2}$.
Such a potential draws much attention from the viewpoint of mathematical physics.
For example, it was shown in Theorem X.11 in \cite{MR0493420}
% M. Reed, B. Simon,
%Methods of modern mathematical physics. II. Fourier analysis, self-adjointness, 
%Academic Press New York-London, 1975.
that the minimal operator $S_{\min}=-\Delta+Br^{-2}$ with the domain $D(S_{\min})=C_0^\infty({\mathbb R}^n \setminus \{0\})$ is essentially self-adjoint if $B \ge 1-(n-2)^2/4$. 
Moreover, it admits the \emph{Friedrichs} extension if $B \ge -(n-2)^2/4$, and the domain of the extension is included in $H^1({\mathbb R}^n)$ if $B > -(n-2)^2/4$, due to the Hardy inequality when $n \ge 3$.
In this direction, there are many works such as \cite{MR742415,MR3906250} 
% Baras, P. and Goldstein, J. A.The heat equation with a singular potential, Trans. Amer.Math. Soc.,284(1984), 121-139. MR85f:35099
%https://www.researchgate.net/publication/327085634_Critical_dissipative_estimate_for_a_heat_semigroup_with_a_quadratic_singular_potential_and_critical_exponent_for_nonlinear_heat_equations
for the heat equation, \cite{MR2003358,MR2106340,MR3959451}
%https://arxiv.org/abs/math/0207152
%https://arxiv.org/pdf/math/0401019.pdf
%https://www.aimsciences.org/article/doi/10.3934/eect.2019022
for the wave and \emph{Schr\"odinger} equations.
However, we do not know any blow-up result for the nonlinear wave equation with the singular inverse square potential, to the best of our knowledge.
In order to prove the blow-up result, we mainly follow the test function method,
so that we need to find a suitable special solution to the adjoint problem.
For this reason, we pose the following requirements on $A$ and $B$:
%In this paper, we consider the equation under the following technical requirements,
\begin{Eq}\label{eq:re}
B>-(n-2)^2/4,\quad 0\leq A<n-1+2\rho,
\end{Eq}
with
\begin{Eq}\label{eq:rho_d}
\rho:=\frac{(2-n)+\sqrt{(n-2)^2+4B}}{2}.
\end{Eq}

To state our result, we introduce a couple of notations.
The critical exponent of \cref{eq:u} and the exact lifespan are
denoted by $p_c(n, A, B)$ and $T_{\varepsilon,A,B}(n, p)$,
respectively. We use $h_S(n,p)$ defined in \cref{eq:hs}
and set $p_S(x)$ to be the positive root of $h_S(x,p)=0$. Such notations are corresponding to the \emph{Strauss} exponent. 
We also use $h_F(n,p)$ defined by
\begin{Eq*}
h_F(n,p):=np-(n+2) 
\end{Eq*}
and set $p_F(x)$ to be the root of $h_F(x,p)=0$. Such notations are corresponding to the \emph{Fujita} exponent.

Now, we are in a position to state our main result in this paper about the upper bound of the lifespan of solutions to \cref{eq:u}.
Firstly, we give a definition of the solution space to \cref{eq:u} which is a reasonable energy space for the \emph{Friedrichs} extension of $S_{\min}$.

\begin{defi}\label{de:u_d}
Assume \cref{eq:re} is satisfied and $\supp(f,g)\subset B(0,1)$ where $B(x,r)$ stands for the ball centered at $x$ with radials $r$. For $f\in H^1(\R^n)$, $g\in L^2(\R^n),$
 we call $u$ a well-posed solution of \cref{eq:u} in $[0,T]\times \R^n$ if 
\begin{Eq*}
&u\in C^1([0,T];L^2(\R^n))\cap C([0,T];H^1(\R^n)), \\
&{|u|^{p}\in L^1(0,T;H^{-1}(\R^n))},\quad \supp u\subset B(0,1+t)
\end{Eq*}
and $u$ satisfies the integral equation
\begin{Eq*}
&\int_{\R^n}
\varepsilon
\kl(g+Ar^{-1}f\kr)\Phi(0)\d x+\int_0^T\int_{\R^n}|u|^p\Phi\d x\d t\\
=&\int_0^T\int_{\R^n}\kl(\nabla u\cdot\nabla\Phi+Br^{-2}u\Phi\kr)\d x\d t-\int_0^T\int_{\R^n}\kl(\partial_tu+Ar^{-1}u\kr)
\partial_t\Phi\d x\d t
\end{Eq*}
for any $\Phi\in \cap_{k=0}^1C_0^k((-\infty,T);H^{1-k}(\R^n))$. 
\end{defi}

\begin{thm} \label{th:M}
If there exists some $u$ which is the solution in the sense of \De{de:u_d} with maximum lifespan $T_{\varepsilon,A,B}(n,p)$
and the initial data $(f,g)$ satisfies
\begin{Eq*}
\int_{\R^n}r^\rho\kl(g(x)+Ar^{-1}f(x)\kr)\d x>0,
\end{Eq*}
then we have that

\noindent{\bf(i):} 
When $\frac{n+\rho}{n+\rho-1}<p<p_S(n+A)$, for some constant $C_1$ depending on $(f,g)$ we have
\begin{Eq}\label{eq:Main_Ss}
T_{\varepsilon,A,B}(n,p)\leq C_1\varepsilon^{\frac{2p(p-1)}{h_S(n+A,p)}}.
\end{Eq}
\noindent{\bf(ii):}
When $p=p_S(n+A)$, for some constant $C_2$ depending on $(f,g)$ we have
\begin{Eq} \label{eq:Main_Sc}
T_{\varepsilon,A,B}(n,p)\leq \exp\kl(C_2\varepsilon^{-p(p-1)}\kr).
\end{Eq}
\noindent{\bf(iii):} 
When $\frac{n+\rho}{n+\rho-1}<p<p_F(n-1+\rho)$, for some constant $C_3$ depending on $(f,g)$ we have
\begin{Eq}\label{eq:Main_Fs}
T_{\varepsilon,A,B}(n,p)\leq C_3\varepsilon^{\frac{p-1}{h_F(n-1+\rho,p)}}.
\end{Eq}
\noindent{\bf(iv):} 
When $p=p_F(n-1+\rho)$, for some constant $C_4$ depending on $(f,g)$ we have
\begin{Eq}\label{eq:Main_Fc}
T_{\varepsilon,A,B}(n,p)\leq \exp\kl(C_4\varepsilon^{1-p}\kr).
\end{Eq}
\end{thm}

\begin{remark}
We require $B>-(n-2)^2/4$ so that the extension of operator $P$ to $x=0$ is unique in our space of solutions. But when $A=0$, we can weaken the condition to $B\geq-(n-2)^2/4$ by modifying the definition of solutions suitably via the domain of the \emph{Friedrichs} extension of $S_{\min}$. Also, when $A=0$, we do not need the requirement $\frac{n+\rho}{n+\rho-1}<p$ because no singularity comes from the damping.
\end{remark}

\begin{remark}
By \Th{th:M}, we tend to believe that the blow-up phenomenon can be considered as a competition between $\emph{Strauss}$ phenomenon and $\emph{Fujita}$ phenomenon. This means the critical exponent $p_c(n,A,B)$ should be equal to $\min\{p_S(n+A),p_F(n-1+\rho)\}$ and the upper bound of the lifespan should be given in a similar way. 
However, we already knew that when $n=2, A=0, B=0,p=2$, the estimate of the lifespan above is not sharp.
So, we believe that when $\frac{n+2\rho-1-A}{2}p=1$, which means \cref{eq:Main_Ss} coincide with \cref{eq:Main_Fs}, there will be another estimate for the lifespan.
\end{remark}

This paper is organized as follows.
In Section 2, we introduce a different type of test functions that solve the conjugate equation of the corresponding homogeneous equation to \eqref{eq:u}.
The blow-up result for the general case is treated in Section 3.
Moreover, when the coefficient of the potential term is relatively small,
we can get a similar result corresponding to the \emph{Strauss} exponent. 
This result is done in section 4,
based on the comparison principle and the explicit representation of
the fundamental solution.

In this paper, 
we denote $x\lesssim y$ and $y\gtrsim x$ if $x\leq Cy$ for some $C>0$, independent of $\varepsilon$, which may change from line to line.
Similarly, $x\approx y$ means that $x\lesssim y\lesssim x$.

\section{Special solutions of the conjugate equation deduced by \cref{eq:u}}

In this section,
firstly we will construct a family of special solutions to 
\begin{Eq}\label{eq:Psi}
P^*\Psi=\partial_t^2\Psi - Ar^{-1}\partial_t \Psi -\Delta \Psi +Br^{-2}\Psi=0,
\end{Eq}
here $P^*$ is the time-space conjugate operator of $P$.
Then,
we will discuss the properties of such $\Psi$,
and use them as the test functions in the next section.

The basic idea is to reshape \cref{eq:Psi} 
and seek solutions only depend on $t$ and $r$.
We consider $\Psi=r^\rho\Phi$ where $\rho$ is defined in \cref{eq:rho_d} which solves
\begin{Eq}\label{eq:rho}
\rho(\rho-1)+(n-1)\rho-B=0.
\end{Eq}
Here we require $B\geq -(n-2)^2/4$ so that $\sqrt{(n-2)^2+4B}$ in \cref{eq:rho_d} makes sense. 

\begin{lemma}
Assume $\Psi$ is spherically symmetric. Then $P^*\Psi(t,x)=0$ if and only if $\Phi(t,x)$ satisfy the equation
\begin{Eq}\label{Phi_s}
\partial_t^2\Phi-Ar^{-1}\partial_t\Phi-\partial_r^2\Phi -(n-1+2\rho)r^{-1}\partial_r \Phi=0,\quad t>0,\quad r>0
\end{Eq}
with above notations.
\end{lemma}

\begin{proof}
We only need to replace $\Psi(t,x)$ by $r^\rho\Phi(t,x)$ 
and multiply $r^{-\rho}$ on both sides. 
For the last two terms in $P^*\Phi$, 
we have 
\begin{Eq*}
r^{-\rho}\partial_r \kl(r^\rho\Phi\kr)&=\rho r^{-1}\Phi+\partial_r\Phi,\\
r^{-\rho}\partial_r^2 \kl(r^\rho\Phi\kr)&=\rho(\rho-1) r^{-2}\Phi+2\rho r^{-1}\partial_r\Phi+\partial_r^2\Phi,\\
r^{-\rho}(-\Delta+Br^{-2})\Psi&=-r^{-\rho}(\partial_r^2+(n-1)r^{-1}\partial_r-Br^{-2})\kl(r^\rho\Phi\kr)\\
&=-\kl(\partial_r^2+(n-1+2\rho)r^{-1}\partial_r\kr)\Phi,
\end{Eq*}
where we used \cref{eq:rho} in the last equality.
Adding the first two terms, we finish the proof.
\end{proof}

Now, we seek a family of homogeneous solutions of \cref{Phi_s} with the form
\begin{Eq*}
\Phi_{\beta}(t,x;\lambda)=(t+r+\lambda)^{-\beta}\phi\kl(\frac{2r}{t+r+\lambda}\kr),\quad \beta\in\R,\lambda\geq 0.
\end{Eq*}

\begin{lemma}\label{le:Q_d}
Let $\Phi_\beta$ be the one defined above. Then $\Phi_\beta$ satisfies \cref{Phi_s} in $Q_\lambda:=\{(t, x) : t+\lambda > |x|\}$ if and only if
\begin{Eq}\label{eq:phi}
0=z(1-z)\phi''(z)+(\gamma-(\alpha+\beta+1)z)\phi'(z)-\alpha\beta\phi(z),\quad z\in(0,1)
\end{Eq}
with the notation
\begin{Eq}
\alpha=\frac{n+A-1+2\rho}{2},\quad \gamma=n-1+2\rho,
\end{Eq}
which will be used through the paper.
\end{lemma}
\begin{proof}
This lemma is verified by some calculations 
and similar to the proof in \cite{ikeda1709life}, 
but for the reader's convenience, we show the details. 
It is obvious that we only need to prove the case $\lambda=0$.
Setting $z=2r(t+r)^{-1}$, 
we have
\begin{Eq*}
\partial_t z =-2r(t+r)^{-2}, \quad \partial_r z=2(t+r)^{-1}-2r(t+r)^{-2}=2t(t+r)^{-2},
\end{Eq*}
then we get
\begin{Eq*}
\partial_t \Phi_\beta =&-\beta(t+r)^{-\beta-1}\phi-2r(t+r)^{-\beta-2}\phi',\\
\partial_t^2 \Phi_\beta %=&\beta(\beta+1)(t+r)^{-\beta-2}\phi+2\beta r\phi'(t+r)^{-\beta-3}\\
%&+2(\beta+2)r\phi'(t+r)^{-\beta-3}+4r^2(t+r)^{-\beta-4}\phi''\\
=&\beta(\beta+1)(t+r)^{-\beta-2}\phi+4(\beta+1) r\phi'(t+r)^{-\beta-3}+4r^2(t+r)^{-\beta-4}\phi'',\\
\partial_r \Phi_\beta =&-\beta(t+r)^{-\beta-1}\phi+2t(t+r)^{-\beta-2}\phi',\\
\partial_r^2 \Phi_\beta %=&\beta(\beta+1)(t+r)^{-\beta-2}\phi-2\beta r\phi'(t+r)^{-\beta-3}\\
%&-2(\beta+2)r\phi'(t+r)^{-\beta-3}+4r^2(t+r)^{-\beta-4}\phi''\\
=&\beta(\beta+1)(t+r)^{-\beta-2}\phi-4(\beta+1) t\phi'(t+r)^{-\beta-3}+4t^2(t+r)^{-\beta-4}\phi''.\\
\end{Eq*}
Then
\begin{Eq*}
0=&\kl(\partial_t^2-Ar^{-1}\partial_t-\partial_r^2-(n-1+2\rho)r^{-1}\partial_r\kr)\Phi_\beta\\
=&4(r-t)(t+r)^{-\beta - 3}\phi''+(2Ar  -2(n-1+2\rho)t+ 4(\beta+1)r)r^{-1}(t+r)^{-\beta-2}\phi'\\
&+(A\beta +  (n-1+2\rho)\beta)r^{-1}(t+r)^{-\beta-1} \phi.
\end{Eq*}
We multiply above equation with $(t+r)^{\beta+2}$ in both sides, using the fact $tr^{-1}=2z^{-1}-1$ and we get
\begin{Eq*}
0=&4\frac{1-tr^{-1}}{1+tr^{-1}}\phi''+(2A -2(n-1+2\rho)tr^{-1}+ 4(\beta+1))\phi'\\
&+\kl(A\beta +  (n-1+2\rho)\beta\kr)\kl(tr^{-1}+1\kr) \phi\\
=&4(z-1)\phi''+\kl(2A + 2(n-1+2\rho)-4(n-1+2\rho)z^{-1}+4(\beta+1)\kr)\phi'\\
&+\kl(A\beta +  (n-1+2\rho)\beta\kr)\kl(2z^{-1}\kr) \phi\\
=&-4z^{-1}\bigg(z(1-z)\phi''+\kl(n-1+2\rho-\kl(\frac{n+A-1+2\rho}{2}+\beta+1\kr)z\kr)\phi'\\
&-\frac{A\beta +  (n-1+2\rho)\beta}{2}\phi\bigg)\\
=&-4z^{-1}\kl(z(1-z)\phi''(z)+(\gamma-(\alpha+\beta+1)z)\phi'(z)-\alpha\beta\phi(z)\kr),
\end{Eq*}
which finishes the proof.
\end{proof}
\begin{coro}\label{co:F_d}
When $A<n-1+2\rho$, the equation \cref{eq:phi} has a special solution
\begin{Eq*}
\phi(z)=F\kl(\alpha,\beta,\gamma;z\kr),
\end{Eq*} 
where $F(\alpha,\beta,\gamma;z)$ is the hypergeometric function given by 
\begin{Eq*}
F(\alpha,\beta,\gamma;z)=\frac{\Gamma(\gamma)}{\Gamma(\alpha)\Gamma(\gamma-\alpha)}\int_0^1 s^{\alpha-1}(1-s)^{\gamma-\alpha-1}(1-zs)^{-\beta}\d t
\end{Eq*}
for $\gamma>\alpha>0$. 
\end{coro}
This corollary just followed from the property of the hypergeometric differential equations.
For more properties of hypergeometric differential equations and hypergeometric functions,
we refer the readers to, e.g., \cite{MR0350075}.

Now, 
we are going to discuss some properties of such $\Psi_\beta=r^\rho\Phi_\beta$.
\begin{lemma}\label{le:Psi_p}
For any $\beta\in\R$ and $(t,x)\in Q_\lambda$ with $Q_\lambda$ defined in \Le{le:Q_d}, we have
\begin{Eq}\label{eq:Psi_r}
\partial_t \Psi_{\beta}(t,x;\lambda)=-\beta\Psi_{\beta+1}(t,x;\lambda).
\end{Eq}
Moreover, for every $(t,x)\in Q_\lambda$, we have
\begin{Eq}\label{eq:Psi_e}
\Psi_{\beta}(t,x;\lambda)\approx r^{\rho}(\lambda+t)^{-\beta}\cdot \begin{cases}
1, & \beta<\gamma-\alpha,\\
1-\ln\kl(1-\frac{r}{\lambda+t}\kr),& \beta=\gamma-\alpha,\\
\kl(1-\frac{r}{\lambda+t}\kr)^{\gamma-\alpha-\beta},&\beta>\gamma-\alpha.
\end{cases}
\end{Eq}
\end{lemma}
\begin{proof}
This lemma is similar to \cite[Lemma 3.2]{ikeda1709life} but we give the details here. 
Firstly for \cref{eq:Psi_r}, with $z=\frac{2r}{t+r+\lambda}$
we have
\begin{Eq*}
\partial_t \Psi_{\beta}(t,x;\lambda)&=r^\rho(t+r+\lambda)^{-\beta-1}(-\beta-z\partial_z) F\kl(\alpha,\beta,\gamma;z\kr),\\
\Psi_{\beta+1}(t,x)&=r^\rho(t+r+\lambda)^{-\beta-1} F\kl(\alpha,\beta+1,\gamma;z\kr).\\
\end{Eq*}
Using the properties 
\begin{Eq*}
\partial_z F\kl(\alpha,\beta,\gamma;z\kr)&=\frac{\alpha\beta}{\gamma} F\kl(\alpha+1,\beta+1,\gamma+1;z\kr),\\
\frac{\alpha z}{\gamma} F\kl(\alpha+1,\beta+1,\gamma+1;z\kr)&=F\kl(\alpha,\beta+1,\gamma;z\kr)- F\kl(\alpha,\beta,\gamma;z\kr)
\end{Eq*}
of hypergeometric functions, see e.g. \cite[Section 9.2]{MR0350075}, we get \cref{eq:Psi_r} for any $n\geq 1$.
For \cref{eq:Psi_e}, we only need to estimate $F\kl(\alpha,\beta,\gamma;z\kr)$ here.
Firstly we notice that $1-z\approx1-\frac{r}{\lambda+t}$ by its expression.
Also, by the expression of $F$,
it is obvious that $F\geq C>0$ for some $C$ when $0<z<1$.
When $0<z\leq 1/2$, for $0<s<1$ we have $(1-zs)^{-\beta}\approx 1$ so we get the result.
As for $z>1/2$, we have
\begin{Eq*}
F(\alpha,\beta,\gamma;z)&\approx\int_0^1 t^{\alpha-1}(1-t)^{\gamma-\alpha-1}(1-zt)^{-\beta}\d t\\
&=\int_{1/2}^1 t^{\alpha-1}(1-t)^{\gamma-\alpha-1}(1-zt)^{-\beta}\d t+O(1)\\
&\approx\int_{1/2}^1 (1-t)^{\gamma-\alpha-1}(1-zt)^{-\beta}\d t+O(1).
\end{Eq*}
Using the change of variable $t=1-(1-z)s$, 
we continue the calculation and get
\begin{Eq*}
F(\alpha,\beta,\gamma;z)&\approx(1-z)^{\gamma-\alpha-\beta}\int_{0}^{(1-z)^{-1}/2} s^{\gamma-\alpha-1}(1+zs)^{-\beta}\d s+O(1)\\
%&=(1-z)^{\gamma-\alpha-\beta}\int_{1}^{(1-z)^{-1}/2} s^{\gamma-\alpha-1}(1+zs)^{-\beta}\d s+O(1)\\
&\approx(1-z)^{\gamma-\alpha-\beta}\int_{1}^{(1-z)^{-1}/2} s^{\gamma-\alpha-\beta-1}\d s+O(1).
\end{Eq*}
By a fundamental calculation of this integral with different $\beta$, we finish our proof of \cref{eq:Psi_e}.
\end{proof}

\section{Proof of the blow-up phenomenon}
In this section we use $\Psi_\beta$ defined in the last section to consider the upper bound of the lifespan of solutions to \cref{eq:u} and its dependence of $\varepsilon$ under the condition \cref{eq:re}. 
\subsection{Preliminaries for showing blowup phenomenon}
To begin with, we introduce a cut-off function $\eta(t)\in C_0^\infty((-\infty,1))$ satisfies
\begin{Eq*}
\eta(t)=\begin{cases}
1,&0\leq s\leq\frac{1}{2},\\
\mbox{non-increasing},& \frac{1}{2}<s<1
\end{cases},
\end{Eq*}
and $\eta_T(t):=\eta(t/T)$ with arbitrary $T>0$.

Now, we want to multiply a test function to both sides of equation \cref{eq:u} and use integration by parts. But, instead of multiplying some $\psi$ satisfies $P^*\psi=0$ as many proofs do, we further multiply $\eta_T(t)^{2p'}$ on both sides, then we will get that 
\begin{lemma}\label{le:u_r}
Let $u$ be a solution of \cref{eq:u} in the sense of \De{de:u_d} with maximum lifespan $T_\varepsilon$, then for any $1\leq T<T_\varepsilon$, we have 
\begin{Eq}\label{eq:G_r}
\varepsilon C_{f,g}(\lambda,\beta_1)+\int_0^TG_{\beta_1}(t;T)\d t\lesssim \kl(\int_{{T/2}}^TG_{\beta}(t;T)\d t\kr)^{\frac{1}{p}}S(T;\beta_1,\beta),
\end{Eq}
for any $\beta,\beta_1\in\R$ and $\lambda>1$,where
\begin{Eq*}
C_{f,g}(\lambda,\beta):=&\int_{\R^n} \kl(g(x)+Ar^{-1}f(x)\kr)\Psi_{\beta}(0,x;\lambda)-f(x)\partial_t\Psi_{\beta}(0,x;\lambda)\d x,\\
G_\beta(t;T):=&\int_{\R^n}|u|^p\eta_T^{2p'}\Psi_\beta\d x,\\
S(T;\beta_1,\beta):=&T^{-2}\kl(\int_{T/2}^T\int_{B(0,1+t)} \Psi_{\beta_1}^{p'}\Psi_{\beta}^{-\frac{p'}{p}}\d x\d t\kr)^{\frac{1}{p'}}\\
&+T^{-1}\kl(\int_{T/2}^T\int_{B(0,1+t)}r^{-p'}\Psi_{\beta_1}^{p'}\Psi_{\beta}^{-\frac{p'}{p}}\d x\d t\kr)^{\frac{1}{p'}}\\
&+T^{-1}\kl(\int_{T/2}^T\int_{B(0,1+t)}\kl|\partial_t\Psi_{\beta_1}\kr|^{p'}\Psi_{\beta}^{-\frac{p'}{p}}\d x\d t\kr)^{\frac{1}{p'}}.
\end{Eq*}
\end{lemma}
\begin{proof}
To begin with, for any given $\beta_1\in\R$, we want to choose $\eta_T^{2p'}\Psi_{\beta_1}$ to be the test function. Noticing that $\eta_T^{2p'}\in C_0^\infty((-\infty,T))$ and $\supp_x u\subset B(0,1+t)$, we only need to show that $\Psi_{\beta_1}$ belongs to $C^k(\R;H^{1-k}(B(0,1+t))$ for $k=0,1$.

By \cref{eq:Psi_e} with $\lambda>1$, we get that
\begin{Eq*}
\Psi_{\beta_1}\lesssim C_t r^\rho,
\end{Eq*}
in $B(0,1+t)$ with some constant $C_t$ depending on $t$. Since $2\rho+n-1>1$ by \cref{eq:rho_d}, its easy to see that $\|\Psi_{\beta_1}\|_{L_x^2(B(0,t+1))}$ is bounded by some constant depending on $t$.
The same estimate also holds for $\|\partial_t \Psi_{\beta_1}\|_{L_x^2(B(0,t+1))}$ because of \cref{eq:Psi_r}. As for the space derivatives, since $\Psi_{\beta_1}$ is spherically symmetric, we only need to consider $\partial_r \Psi_{\beta_1}$. However, the worst singularity will be $r^{\rho-1}$ with $2(\rho-1)+n-1>-1$. So we also have that $\|\partial_r \Psi_{\beta_1}\|_{L_x^2(B(0,t+1))}$ is bounded by some constant depending on $t$.
%
%
%
%First of all, we remark that
%$\Psi_{\beta_1} \in \displaystyle \bigcap_{k=0}^1 C_0^k ((-\infty,T); H^{1-k}(B(0,1+t))$.
%In view of \eqref{eq:Psi_e}, we get
%\begin{Eq*} 
%\int_{B(0,t+1)}\kl|\Psi_{\beta_1}\kr|^{2} \d x
%\lesssim &\int_{B(0,t+1)}r^{2\rho} (t+\lambda)^{-2\beta_1} 
%\Theta_{\beta_1}(t,r;\lambda)^2 \d x \\
%\lesssim &(t+1)^{2\rho+n-1-2\beta_1} \int_0^{t+1}
%\Theta_{\beta_1}(t,r;\lambda)^2 \d r,
%\end{Eq*}
%because $2\rho+n-1=1+\sqrt{(n-2)^2+4B} \,(>1)$.
%Here we put
%$$
%\Theta_{\beta}(t,r;\lambda)
%=\begin{cases}
%1, & \beta<\gamma-\alpha,\\
%1-\ln\kl(1-\frac{r}{\lambda+t}\kr),& \beta=\gamma-\alpha,\\
%\kl(1-\frac{r}{\lambda+t}\kr)^{\gamma-\alpha-\beta},&\beta>\gamma-\alpha.
%\end{cases}
%$$
%A direct computation shows that the last quantity is bounded by some
%constant depending on $T$.
%We see from \eqref{eq:Psi_r} that $\displaystyle \int_{B(0,t+1)}\kl|\partial_t\Psi_{\beta_1}\kr|^{2} \d x$ has the same type of an upper bound.
%In addition, it follows that 
%\begin{Eq*}
%\partial_{x_j} \Psi_{\beta}(t,x;\lambda)=r^{\rho-1} (t+r+\lambda)^{-\beta-1}\nu_j
%( \rho(t+r+\lambda)-\beta r +(t+\lambda)-z\partial_z) F\kl(\alpha,\beta,\gamma;z\kr),
%\end{Eq*}
%where $\nu_j=x_j/r$. Therefore, we obtain
%\begin{Eq*}
%\kl| \nabla \Psi_{\beta}(t,x;\lambda)\kr|
%\lesssim r^{\rho-1} (t+\lambda)^{-\beta} \kl( \Theta_{\beta}(t,r;\lambda)+\Theta_{\beta+1}(t,r;\lambda) \kr),
%\end{Eq*}
%which leads to the boundedness of $\displaystyle \int_{B(0,t+1)}\kl|\nabla\Psi_{\beta_1}\kr|^{2} \d x$. 
Hence, we have shown that $\Psi_{\beta_1}$ belongs to the function space for test functions.

Using \De{de:u_d} and making the integration by parts again, we find that 
%Using the integration by parts, we find that 
\begin{Eq*}
& \int_{\R^n} \varepsilon  (g+Ar^{-1}f ) \eta_T^{2p'}(0)\Psi_{\beta_1}(0) \d x
 % -f (\partial_t \Psi_{\beta_1})(0) \d x
+\int_0^T\int_{\R^n} |u|^p \eta_T^{2p'}\Psi_{\beta_1}\d x\d t \\
%=&\int_0^T\int_{\R^n} (Pu) \eta_T^{2p'}\Psi_{\beta_1}\d x\d t\\
=
&\int_0^T\int_{\R^n}\kl(\nabla u\cdot\nabla \kl( \eta_T^{2p'} \Phi_{\beta_1} \kr)
+Br^{-2}u \,  \eta_T^{2p'} \Phi_{\beta_1} \kr) \d x\d t
\\
& -\int_0^T\int_{\R^n} \kl(\partial_t u+Ar^{-1}u\kr)
\partial_t \kl(  \eta_T^{2p'} \Phi_{\beta_1} \kr) \d x\d t
\\
=
& \int_0^T\int_{\R^n} u P^*\kl(\eta_T^{2p'}\Psi_{\beta_1}\kr)\d x\d t
-\kl.\int_{\R^n} u \, \partial_t\kl(\eta_T^{2p'}\Psi_{\beta_1}\kr) \d x\kr|_{t=0}^T.
%=&\kl.\int_{\R^n} \eta_T^{2p'}\Psi_{\beta_1}\partial_t u-u\partial_t\kl(\eta_T^{2p'}\Psi_{\beta_1}\kr)+Ar^{-1}u\eta_T^{2p'}\Psi_{\beta_1} \d x\kr|_{t=0}^T\\
%&+\int_0^T\int_{\R^n} u P^*\kl(\eta_T^{2p'}\Psi_{\beta_1}\kr)\d x\d t.
\end{Eq*}
%For the first term in right-hand side, noticing that 
Noticing that
$\eta_T(T)=\partial_t\eta_T(T)=\partial_t\eta_T(0)=0$ and $\eta_T(0)=1$,
we can simplify the above identity as
\begin{Eq*}
 \varepsilon C_{f,g}(\lambda,\beta_1)+\int_0^TG_{\beta_1}(t;T)\d t 
= \int_0^T\int_{\R^n} u P^*\kl(\eta_T^{2p'}\Psi_{\beta_1}\kr)\d x\d t.
\end{Eq*}

%For the first term in the left-hand side, 
%we find that this term equals $-\varepsilon C_{f,g}(\lambda,\beta_1)$.

%As for the second term in the right-hand side, noticing that 
Since $P^*\Psi_{\beta_1}=0$ and $\eta_T$ only depends on $t$, we have
\begin{Eq*}
\int_0^T\int_{\R^n} u P^*\kl(\eta_T^{2p'}\Psi_{\beta_1}\kr)\d x\d t=&\int_0^T\int_{\R^n} u \Psi_{\beta_1}  \partial_t^2\kl(\eta_T^{2p'}\kr)\d x\d t\\
&-\int_0^T\int_{\R^n} u \Psi_{\beta_1}   Ar^{-1}\partial_t \kl(\eta_T^{2p'}\kr) \d x\d t\\
&+\int_0^T\int_{\R^n}  2u\partial_t\Psi_{\beta_1} \partial_t\kl(\eta_T^{2p'}\kr)\d x\d t.
\end{Eq*}
For the first part, noticing that $2p'-2=\frac{2p'}{p}$, $\supp u(t,\cdot)\subset B(0,t+1)$ and for every $k\geq 1$, $\kl.\eta_T^{(k)}(t)\kr|_{0\leq t\leq T/2}=0$ and $\kl|\eta_T^{(k)}(t)\kr|_{0\leq t\leq T}\lesssim C_k$ with some constant $C_k$, we use the chain rule and get that
\begin{Eq*}
&\kl|\int_0^T\int_{\R^n} u \Psi_{\beta_1}  \partial_t^2\kl(\eta_T^{2p'}\kr)\d x\d t\kr|\\
\lesssim& T^{-2}\int_{T/2}^T\int_{B(0,t+1)} \kl|u \Psi_{\beta_1}\eta_T''\eta_T^{2p'-1}\kr|+\kl|u \Psi_{\beta_1} \eta'^2\eta_T^{2p'-2} \kr|\d x\d t\\
\lesssim& T^{-2}\kl(\int_{T/2}^T\int_{B(0,t+1)} |u|^p \Psi_{\beta} \eta_T^{2p'}\d x\d t\kr)^{\frac{1}{p}}\kl(\int_{T/2}^T\int_{B(0,t+1)} \Psi_{\beta_1}^{p'}\Psi_{\beta}^{-\frac{p'}{p}}\d x\d t\kr)^{\frac{1}{p'}}.
\end{Eq*}
Estimates for the rest two parts go similarly, so we omit them and finish the proof.

\end{proof}

To use the inequality above, we need to ask $C_{f,g}(\lambda,\beta)>0$ for some $\lambda$, for this purpose we give the following lemma.
\begin{lemma}\label{le:da_r}
For any given $\beta$, if $f\in H^1(\R^n)$, $g\in L^2(\R^n)$ with compact support and
\begin{Eq*}
\int_{\R^n}r^\rho\kl(g(x)+Ar^{-1}f(x)\kr)\d x>0,
\end{Eq*}
then there exists a $\lambda_\beta>1$ such that for any $\lambda>\lambda_\beta$, $C_{f,g}(\lambda,\beta)>0$.
\end{lemma}
\begin{proof}
Based on \Le{le:Psi_p}, when $\lambda>1$, $t=0$ and $r<1$, we have
\begin{Eq*}
\Psi_\beta(0,x;\lambda)\approx r^\rho\lambda^{-\beta},\qquad
\partial_t \Psi_\beta(0,x;\lambda)\lesssim r^\rho\lambda^{-\beta-1}.
\end{Eq*}
Noticing that $r^{\rho-1}\in L_{loc}^2(\R^n)$, so all integrates appeared in $C_{f,g}(\lambda,\beta)$ is finite. 
Then, this lemma is just an obvious conclusion.
\end{proof}

Next, we want to use \Le{le:u_r} to estimate $\int G_\beta(t;T)\d t$ for some $\beta$.
So we need to estimate the $S(T;\beta_1,\beta)$.

Here we mention that when we want to estimate the relation between $\int_{T/2}^T G_\beta(t;T)\d t$ and the initial data, we should let $S(T;\beta_1,\beta)$ be small. 
And when we want to estimate the relation between $\int_{T/2}^T G_\beta(t;T)\d t$ and$\int_{0}^T G_\beta(t;T)\d t$, we should choose $\beta_1=\beta$. Now we give the following lemma.
\begin{lemma}\label{le:S_e}
Assume $T\geq 1$, $\lambda>1$, $p>\frac{n+\rho}{n+\rho-1}$ and $\beta< \gamma-\alpha$.
When $\beta_1=0$,  
we have that 
\begin{Eq*}
S(T;0,\beta)\lesssim T^{\frac{\rho+n+1}{p'}-2+\frac{\beta}{p}}.
\end{Eq*}
When $\beta_1>\gamma-\alpha-\frac{1}{p}$,  
we have that 
\begin{Eq*}
S(T;\beta_1,\beta)\lesssim T^{\frac{\rho+n}{p'}-\frac{n-A+2\rho+1}{2}+\frac{\beta}{p}}.
\end{Eq*}
When $\beta_1=\beta\neq 0$, 
we have that 
\begin{Eq*}
S(T;\beta,\beta)\lesssim& T^{\frac{\rho+n}{p'}-\beta-2+\frac{\beta}{p}}S_0(T)^{\frac{1}{p'}},\\
S_0(T):= &
\begin{cases}
T,&\beta<\gamma-\alpha-\frac{1}{p},\\
T\ln(T+1),&\beta=\gamma-\alpha-\frac{1}{p},
\end{cases}
\end{Eq*} 
\end{lemma}
\begin{proof}
We begin with the $\beta_1=0$ situation, for the first term in $S(T;0,\beta)$, when $\beta< \gamma-\alpha$ we calculate
\begin{Eq*} 
\int_{B(0,t+1)}\Psi_{0}^{p'}\Psi_{\beta}^{-\frac{p'}{p}}\d x\lesssim &\int_{B(0,t+1)}r^{\rho}(t+\lambda)^{(-\beta)\kl(-\frac{p'}{p}\kr)}\d x\\
\lesssim &(t+1)^{\beta\frac{p'}{p}}\int_0^{t+1}r^{\rho+n-1}\d r\\
\approx& (t+1)^{\rho+n+\beta\frac{p'}{p}}.
\end{Eq*}
Then we have that 
\begin{Eq*}
T^{-2}\kl(\int_{T/2}^T\int_{B(0,t+1)}\Psi_{0}^{p'}\Psi_{\beta}^{-\frac{p'}{p}}\d x\d t\kr)^{\frac{1}{p'}}
\lesssim &T^{-2}\kl(\int_{T/2}^T(t+1)^{\rho+n+\beta\frac{p'}{p}}\d t\kr)^{\frac{1}{p'}}\\
\lesssim &T^{\frac{\rho+n+1}{p'}-2+\frac{\beta}{p}}.
\end{Eq*}
Although the estimate for the second term goes similarly, we mention that we need to require $p>\frac{n+\rho}{n+\rho-1}$ such that there is no singularity  in the $r$-integral. As for the third term, it vanishes since $\partial_t\Psi_0=0$ by \cref{eq:Psi_r}. Now we finish the proof for this situation $\beta_1=0$.

When we consider $\beta_1>\gamma-\alpha-\frac{1}{p}$, the estimate for the first two terms is almost the same as before.
As for the third term, we use \Le{le:Psi_p} and calculate
\begin{Eq*} 
&\int_{B(0,t+1)}\kl|\partial_t\Psi_{\beta_1}\kr|^{p'}\Psi_{\beta}^{-\frac{p'}{p}}\d x\\
\lesssim &\int_{B(0,t+1)}r^{\rho}(t+\lambda)^{(-\beta_1-1)p'-(-\beta)\frac{p'}{p}}\kl(1-\frac{r}{\lambda+t}\kr)^{(\gamma-\alpha-\beta_1-1)p'}\d x\\
\lesssim &(t+1)^{\rho+n-1+(-\beta_1-1)p'+\beta\frac{p'}{p}}\int_0^{t+1}\kl(1-\frac{r}{\lambda+t}\kr)^{(\gamma-\alpha-\beta_1-1)p'}\d r.
\end{Eq*}
Based on the assumption on $\beta_1$, we have
\begin{Eq*}
\int_0^{t+1}\kl(1-\frac{r}{\lambda+t}\kr)^{(\gamma-\alpha-\beta_1-1)p'}\d r\lesssim(t+1)^{-(\gamma-\alpha-\beta_1-1)p'}.
\end{Eq*}
Here we mention that $\alpha=\frac{n+A-1+2\rho}{2}$, $\gamma=n-1+2\rho$, then we get the result by a simple calculation. 

Next, we consider $\beta_1=\beta\neq 0$, similarly to the above, we only consider the third term.
Here when $\beta<\gamma-\alpha-1$, we have
\begin{Eq*}
\int_0^{t+1}1\d r\lesssim(t+1)\equiv S_0(t+1).
\end{Eq*}
When $\beta=\gamma-\alpha-1$,  we have
\begin{Eq*}
\int_0^{t+1}\kl(1-\ln\kl(1-\frac{r}{t+\lambda}\kr)\kr)^{p'}\d r\lesssim(t+1)\equiv S_0(t+1).
\end{Eq*}
When $\gamma-\alpha-1<\beta<\gamma-\alpha-\frac{1}{p}$,  we have
\begin{Eq*}
\int_0^{t+1}\kl(1-\frac{r}{\lambda+t}\kr)^{(\gamma-\alpha-\beta-1)p'}\d r\lesssim(t+1)\equiv S_0(t+1).
\end{Eq*}
When $\beta=\gamma-\alpha-\frac{1}{p}$,  we have
\begin{Eq*}
\int_0^{t+1}\kl(1-\frac{r}{\lambda+t}\kr)^{(\gamma-\alpha-\beta-1)p'}\d r\lesssim(t+1)\ln(t+2)\equiv S_0(t+1).
\end{Eq*}
Then we get the results same as before.
\end{proof}

We mention that once we get the estimate of $\int_{{T/2}}^TG_{\beta}(t;T)\d t$, 
we get the estimate of $\int_{0}^TG_{\beta}(t;T)\d t$. 
One version is that the former one is smaller than the latter one. 
Also, a direct application of binary decomposition can be used here because $\eta_T$ increases with $T$ then $G_\beta(t;T)$ increases with $T$.
However, in this paper, we choose to use ODI(ordinary differential inequality) to expose this relation. 
For this purpose, we introduce an auxiliary function.

\begin{lemma}\label{le:H_d}
For any given $\beta\in\R$, if we define 
\begin{Eq*}
H_\beta(t):=\int_0^T\int_{t/2}^tG_\beta(s;t)\d s\frac{\d t}{t},
\end{Eq*}
then
\begin{Eq*}
TH_\beta'(T)=\int_{T/2}^TG_\beta(t;T)\d t,\qquad H_\beta(t)\leq \int_{0}^TG_\beta(t;T)\d t.
\end{Eq*}
\end{lemma}
\begin{proof}
The first equality is obvious, so we only focus on the inequality. Here we exchange the order of integration, then
\begin{Eq*}
H_\beta(T)=&\int_0^T\int_{s}^{\min\{2s,T\}}\frac{1}{t}G_\beta(s;t)\d t\d s\\
&\leq \int_0^T\int_{s}^{2s}\frac{1}{t}G_\beta(s;T)\d t\d s\leq\int_0^TG_\beta(s;T)\d s.
\end{Eq*}
This finishes the proof.
\end{proof}

\subsection{Blow-up phenomenon}
Based on \cref{eq:G_r} and \Le{le:S_e}, we find that the different choices of $\beta_1$ lead to different kind of lower bounds of $\int_{T/2}^T G_\beta(t;T)\d t$. To start with, we fix $\beta<\gamma-\alpha-\frac{1}{p}$.

Firstly we choose $\beta_1>\gamma-\alpha-\frac{1}{p}$ and show the blow-up result corresponding to the \emph{Strauss} exponent.
Using \Le{le:u_r}, \ref{le:da_r} and \ref{le:S_e}, under their condition we have
\begin{Eq*}
\varepsilon \lesssim& T^{\frac{\rho+n}{p'}-\frac{n-A+2\rho+1}{2}+\frac{\beta}{p}} \kl(\int_{T/2}^T G_\beta(t;T)\d t\kr)^\frac{1}{p},\\
\int_{0}^T G_\beta(t;T)\d t\lesssim& T^{\frac{\rho+n+1}{p'}-\beta-2+\frac{\beta}{p}}  \kl(\int_{T/2}^T G_\beta(t;T)\d t\kr)^\frac{1}{p},
\end{Eq*}
respectively. Mixing them and we get that
\begin{Eq}\label{eq:GT_r_S}
\varepsilon^p T^{n+\rho-\frac{n+A-1}{2}p-\beta}
\lesssim\int_{T/2}^T G_\beta(t;T)\d t
\lesssim T^{n+\rho-1-\frac{2}{p-1}-\beta},
\end{Eq}
which means
\begin{Eq*}
T\lesssim \varepsilon^{\frac{2p(p-1)}{h_S(n+A,p)}}
\end{Eq*}
for any $p<p_S(n+A)$. Because this relation must be satisfied for any $1\leq T<T_\varepsilon$, we finish the proof of \cref{eq:Main_Ss}.

Then, we choose $\beta_1=0$ and show the blow-up result corresponding to the \emph{Fujita} exponent.
Using these lemma again, under their condition we have
\begin{Eq*}
\varepsilon \lesssim& T^{\frac{\rho+n+1}{p'}-2+\frac{\beta}{p}} \kl(\int_{T/2}^T G_\beta(t;T)\d t\kr)^\frac{1}{p},\\
\int_{0}^T G_\beta(t;T)\d t\lesssim& T^{\frac{\rho+n+1}{p'}-\beta-2+\frac{\beta}{p}} \kl(\int_{T/2}^T G_\beta(t;T)\d t\kr)^\frac{1}{p}
\end{Eq*}
respectively. Mixing them and we get that
\begin{Eq*}
\varepsilon^p T^{(n+\rho-1)(1-p)+2-\beta}
\lesssim\int_{T/2}^T G_\beta(t;T)\d t
\lesssim T^{n+\rho-1-\frac{2}{p-1}-\beta},
\end{Eq*} 
which means
\begin{Eq*}
T\lesssim \varepsilon^{\frac{p-1}{h_F(n-1+\rho,p)}}
\end{Eq*}
for any $p<p_F(n-1+\rho)$. Similarly to the above we finish the proof of \cref{eq:Main_Fs}.

Next, we are going to consider the `critical' situation.
As before, we begin with the estimate related to \emph{Strauss} exponent.
When $p=p_S(n+A)$, we have 
\begin{Eq*}
n+\rho-\frac{n+A-1}{2}p=n+\rho-1-\frac{2}{p-1}=\gamma-\alpha-\frac{1}{p}.
\end{Eq*}
Then from \cref{eq:GT_r_S} we can only get that
\begin{Eq}%\label{eq:GT_r_S}
\varepsilon^p T^{\gamma-\alpha-\frac{1}{p}-\beta}
\lesssim\int_{T/2}^T G_\beta(t;T)\d t
\lesssim T^{\gamma-\alpha-\frac{1}{p}-\beta}.
\end{Eq}

This suggests us to choose $\beta=\gamma-\alpha-\frac1p$ and treat the integral of $G_\beta(t,T)$  more precisely.
%To analysis the behavior of $u$ more precisely, we naturally choose $\beta=\gamma-\alpha-\frac{1}{p}$. 
Using \Le{le:S_e} again and under above condition we get
\begin{Eq*}
\varepsilon \lesssim& \kl(\int_{T/2}^T G_\beta(t;T)\d t\kr)^\frac{1}{p},\\
\int_{0}^T G_\beta(t;T)\d t\lesssim& \kl(\int_{T/2}^T G_\beta(t;T)\d t\kr)^\frac{1}{p}\ln(T+1)^{\frac{1}{p'}}.
\end{Eq*} 
Here we use $H_\beta$ defined in \Le{le:H_d} and we get that for any $T\geq 2$
\begin{Eq*}
\varepsilon^p \lesssim TH_\beta'(T),\qquad 
H_\beta(T)^p\ln(T)^{1-p}\lesssim TH_\beta'(T).
\end{Eq*}
We introduce $s=\ln(T)$ and $I(s)=H(T)$ where $H'(T)=T^{-1}I'(s)$, then above ODI system equivalents to
\begin{Eq*}
\varepsilon^p \lesssim I'(s),\qquad 
I(s)^p s^{1-p}\lesssim I'(s).
\end{Eq*}
This is a classical ODI system which must blow-up in finite time, using the scaling method we know that $I(s)$ must blow up before $C\varepsilon^{-p(p-1)}$ with some constant $C$, by the relation between $s$ and $T$ we finish the proof of \cref{eq:Main_Sc}.

For the $p=p_F(n-1+\rho)$ situation, similarly we have
\begin{Eq*}
(n+\rho-1)(1-p)+2=n+\rho-1-\frac{2}{p-1}=0.
\end{Eq*}
Since $\gamma-\alpha>0$, we choose $\beta=0$ and get
\begin{Eq*}
\varepsilon \lesssim& \kl(\int_{T/2}^T G_\beta(t;T)\d t\kr)^\frac{1}{p},\\
\int_{0}^T G_\beta(t;T)\d t\lesssim& \kl(\int_{T/2}^T G_\beta(t;T)\d t\kr)^\frac{1}{p}.
\end{Eq*}
Similar as before,
with $s=\ln(T)$ and $I(s)=H(T)$ we get
\begin{Eq*}
\varepsilon^p \lesssim I'(s),\qquad I(s)^p \lesssim I'(s).
\end{Eq*}
Using scaling method again we know that $I(s)$ must blow up before $C\varepsilon^{-(p-1)}$ with some constant $C$, by the relation between $s$ and $T$ we finish the proof for \cref{eq:Main_Fc}.

\section{Some improvements %when $B$ is small}
on the initial data}
In this section, we want to show a bit more results
by means of the idea from \cite{MR3953034} and comparison principle.
%, although we still require $A\geq 0$.
More precisely, we derive the blow-up results in a slightly different solution space from \De{de:u_d}
given in \De{de:u_d2} below, without assuming that the initial data is compactly supported.

For convenience, we assume 
\begin{Eq*}
f(x)\equiv0,\qquad 0\leq g(x)\not \equiv 0
\end{Eq*}
in this section. Moreover, we assume $r^{\frac{n+A-1}{2}}g(x)\in L_{loc}^\infty(\R^n)$, then we give the definition of solutions.
\begin{defi}\label{de:u_d2}
Assume $(f,g)$ as above. We call $u$ is a mild solution of \cref{eq:u} in $[0,T]\times \R^n$ if $r^{\frac{n+A-1}{2}}u\subset C\kl([0,T];L_{loc}^\infty(\R^n)\kr)$ satisfies
\begin{Eq}\label{eq:u_E}
r^{\frac{n+A-1}{2}}\int_{S^{n-1}}u(t,r\nu)\d\nu=&\int_0^t\int_{|r-t+s|}^{r+t-s}\int_{S^{n-1}}\kl(\frac{\tilde B-B}{\rho^2}u(t,\rho\nu)+|u|^p(t,\rho\nu)\kr)\\
&\phantom{\int_0^t\int_{|r-t+s|}^{r+t-s}\int_{S^{n-1}}}\times K(t-s,r,\rho)\rho^{\frac{n+A-1}{2}}\d\nu\d\rho\d s\\
&+\varepsilon\int_{|r-t|}^{r+t}\int_{S^{n-1}}g(\rho\nu)K(t,r,\rho)\rho^{\frac{n+A-1}{2}}\d\nu\d\rho
\end{Eq}
where $\tilde B=\frac{A^2+2A-(n-1)(n-3)}{4}$ and $K(t,r,\rho)=2^{-1-A}\rho^{-A}(\rho+r-t)^A$ is a bounded function in its integral region.
\end{defi}

We mention that when $g$ and $u$ is %sphere 
spherically symmetric, the equation \cref{eq:u} is equivalent to
\begin{Eq}\label{eq:V}
\begin{cases}
\kl(\partial_t-\partial_r+\frac{A}{2r}\kr)\kl(\partial_t+\partial_r+\frac{A}{2r}\kr)V=F(t,r),\\
V(0,r)=\partial_r V(t,0)=0,\quad V_t(0,r)=\varepsilon \tilde g(r),
\end{cases}
\end{Eq}
where
\begin{Eq*}
V=r^{\frac{n-1}{2}}u,\quad F=r^{\frac{n-1}{2}}\kl(\frac{\tilde B-B}{r^2}u+|u|^p\kr),\quad  \tilde g(r)=r^\frac{n-1}{2}g(r).
\end{Eq*}
Then, it is easy to testify that $\cref{eq:u_E}$ gives an expression of solution in the linear case, which leaves to readers. It also means our definition of solution is reasonable. Now, we are going to give the main result in this section.

\begin{thm}\label{th:U_l}
%With $f$ and $g$ as above,
 %$A\geq 0$ and %$B\leq \frac{A^2+2A-(n-1)(n-3)}{4}$, 
 Assume $A\geq 0$ and
 $-\frac{(n-2)^2}4 \le B\leq \frac{A^2+2A-(n-1)(n-3)}{4}$.
Let $f$ and $g$ be as above. If
%if 
$u$ is a solution in $[0,T]\times \R^n$ in the sense of \De{de:u_d2}, we must have
\begin{Eq}\label{eq:T_l}
T\leq& % C\varepsilon^{\frac{2p(p-1)}{h(n+A,p)}},&\quad& p<p_c(n+A,0,0),\\
C\varepsilon^{\frac{2p(p-1)}{h_S(n+A,p)}},&\quad&
p<p_S(n+A),\\
T\leq& \exp\kl(C\varepsilon^{-p(p-1)}\kr),&\quad& % p=p_c(n+A,0,0),
p=p_S(n+A)
\end{Eq}
for some constant $C$ which does not depend on $\varepsilon$.
\end{thm}
 
To prove this theorem, we introduce a proposition.
\begin{prop}[Lemma 3.2 of \cite{MR3953034}]\label{pr:OII}
Let $C_1$, $C_2>0$, $\alpha$, $\beta\ge 0$, $\kappa\le 1$,
$\varepsilon\in (0,1]$, and $p>1$.
Suppose that $f(y)$ satisfies
$$
f(y)\ge C_1\varepsilon^{\alpha},\quad
f(y)\ge C_2\varepsilon^{\beta}\int_{1}^{y}\left(1-\frac{\eta}{y}\right)
\frac{f(\eta)^{p}}{\eta^{\kappa}}\d \eta,\quad y\ge 1.
$$
Then, $f(y)$ blows up in a finite time $T_*(\varepsilon)$. Moreover,
there exists a constant $C^*=C^*(C_1,C_2,p,\kappa)>0$ such that
$$T_*(\varepsilon)\le \left\{ \begin{array}{ll}
\exp(C^*\varepsilon^{-\{(p-1)\alpha+\beta\}})
& \mbox{if} \hspace{2mm} \kappa=1, \\
C^*\varepsilon^{-\{(p-1)\alpha+\beta\}/(1-\kappa)}
& \mbox{if} \hspace{2mm} \kappa<1.
\end{array} \right.$$
\end{prop}

\begin{proof}[Proof of \Th{th:U_l}.]
Firstly, we notice that the $K(t,r,\rho)$ in \cref{eq:u_E} is nonnegative, and that by \emph{Jensen}'s inequality, we have
\begin{Eq*}
\frac{\tilde B-B}{r^2}\int_{S^{n-1}} u\d\nu+\int_{S^{n-1}}|u|^p\d\nu\geq \kl|S^{n-1}\kr|^{1-p}\kl|\int_{S^{n-1}}u\d\nu\kr|^p 
\end{Eq*}
as long as $ \int_{S^{n-1}}u\d\nu\geq 0$. Since that, we only need to deal with the case $B=\tilde B$ and $u$ and $g$ is spherically symmetric. The remaining cases can be proven by the comparison principle.

Secondly, by the assumption on $g$, we know that there must exist a positive constant $c_0$ and some region $[a,b]$, such that $g\geq c_0$ when $r\in[a,b]$. Without loss of generality we assume $a=1/2,b=1$.
Then, for $t<r<t+1/2$, $t+r>1$, by \cref{eq:u_E} we have
\begin{Eq*}
u(t,r)\gtrsim& \varepsilon r^{-\frac{n+A-1}{2}}\int_{r-t}^{r+t}\rho^{-A}(\rho+r-t)^{A} g(\rho)\rho^\frac{n+A-1}{2}\d\rho\\
\gtrsim& \varepsilon r^{-\frac{n+A-1}{2}}\int_{\frac{1}{2}}^{1}\rho^{\frac{n+A-1}{2}} g(\rho)\d\rho\\
\gtrsim& \varepsilon r^{-\frac{n+A-1}{2}}.
\end{Eq*}
By \cref{eq:u_E} again, for $0<t<2r$, $t-r>1$ we have
\begin{Eq*}
u(t,r)\gtrsim& r^{-\frac{n+A-1}{2}}\int_{0}^t\int_{|r-t+s|}^{r+t-s}\kl(\frac{\rho+r+s-t}{\rho}\kr)^{A}|u|^p(s,\rho)\rho^{\frac{n+A-1}{2}}\d\rho\d s.
\end{Eq*}
Noticing that $\Sigma:=\{(s,\rho):0<\rho-s<1/2,t-r<s+\rho<t+r\}$ is a subset of the integral region, we have
\begin{Eq*}
u(t,r)\gtrsim \varepsilon^p r^{-\frac{n+A-1}{2}}\iint_{\Sigma}(\rho+r+s-t)^{A}\rho^{-\frac{n+A-1}{2}p+\frac{n-A-1}{2}}\d\rho\d s.
\end{Eq*}
Using the change of variable $\xi=s+\rho$, $\eta=s-\rho$, we obtain
\begin{Eq*}
u(t,r)\gtrsim&\varepsilon^p r^{-\frac{n+A-1}{2}}\int_{t-r}^{t+r}\int_{-\frac{1}{2}}^0(\xi+r-t)^{A}(\xi-\eta)^{-\frac{n+A-1}{2}p+\frac{n-A-1}{2}}\d\eta\d\xi\\
\gtrsim&\varepsilon^p r^{-\frac{n+A-1}{2}}\int_{t-r}^{t+r}(\xi+r-t)^{A}\xi^{-\frac{n+A-1}{2}p+\frac{n-A-1}{2}}\d\xi.
\end{Eq*}
Since $t<2r$, we have $t+r>3(t-r)$, so that
\begin{Eq*}
u(t,r)\gtrsim&\varepsilon^pr^{-\frac{n+A-1}{2}}\int_{t-r}^{3(t-r)}(\xi+r-t)^{A}\xi^{-\frac{n+A-1}{2}p+\frac{n-A-1}{2}}\d\xi\\
\gtrsim&\varepsilon^pr^{-\frac{n+A-1}{2}}(t-r)^{-\frac{n+A-1}{2}p+\frac{n-A-1}{2}}\int_{t-r}^{3(t-r)}(\xi+r-t)^{A}\d\xi\\
\gtrsim&\varepsilon^pr^{-\frac{n+A-1}{2}}(t-r)^{-\frac{n+A-1}{2}p+\frac{n+A+1}{2}}.
\end{Eq*}

Since it is so, we set $p^*:=\frac{1}{2}\big((n+A-1)p-(n+A+1)\big)$ and consider
\begin{Eq*}
f(y):=&\inf_{(s,\rho)\in\Omega_y}\rho^{\frac{n+A-1}{2}}(s-\rho)^{p^*}u(s,\rho),\\
\Omega_y:=&\{(s,\rho):0\leq s\leq 2\rho, s-\rho\geq y\}.
\end{Eq*}
By the discussion before, we firstly find that $f(y)\geq C_1\varepsilon^p$ for any $y>1$ and some positive constant $C_1$. In addition, by \cref{eq:u_E} again we find that, for any $(t,r)\in\Omega_y$ with $y\geq 1$, if we set
\begin{Eq*}
\tilde \Omega_{z,\eta}:=\{(s,\rho):\rho\geq\eta,s+\rho\leq 3\eta, s-\rho\geq z\}\subset\Omega_z,\quad z\geq1,\quad\eta\geq1,
\end{Eq*}
we have
\begin{Eq*}
u(t,r)\gtrsim&r^{-\frac{n+A+1}{2}}\iint_{\tilde \Omega_{1,t-r}}\kl(\frac{\rho+r+s-t}{\rho}\kr)^{A}|u|^p(s,\rho)\rho^{\frac{n+A-1}{2}}\d\rho\d s\\
\gtrsim&r^{-\frac{n+A-1}{2}}\iint_{\tilde \Omega_{1,t-r}}\frac{(\rho+r+s-t)^A}{\rho^{p^*+A+1}}\frac{f(s-\rho)^p}{(s-\rho)^{pp^*}}\d\rho\d s.
\end{Eq*}
Here $\rho+r+s-t=2\rho+r-t+(s-\rho)>t-r$ for $(s,\rho)\in \tilde \Omega_{1,t-r}$. Changing the variables by $\eta=s-\rho$, $\rho=\rho$, we have
\begin{Eq*}
u(t,r)\gtrsim&\frac{(t-r)^A}{r^{\frac{n+A-1}{2}}}\int_1^{t-r}\int_{t-r}^{\frac{3(t-r)-\eta}{2}}\frac{1}{\rho^{p^*+A+1}}\frac{f(\eta)^p}{\eta^{pp^*}}\d\rho\d\eta\\
\gtrsim&\frac{1}{r^{\frac{n+A-1}{2}}(t-r)^{p^*}}\int_1^{t-r}\kl(1-\frac{\eta}{t-r}\kr)\frac{f(\eta)^p}{\eta^{pp^*}}\d\eta.
\end{Eq*}
This shows 
\begin{Eq*}
f(y)\geq C_2 \int_1^{y}\kl(1-\frac{\eta}{y}\kr)\frac{f(\eta)^p}{\eta^{pp^*}}d\eta
\end{Eq*}
with $y>1$ and some constant $C_2$. By \Pr{pr:OII} with $\kappa=pp^*=\frac{h(n+A,p)}{2}+1$ and the relation between $u$ and $f$, we finish the proof.

\end{proof}

\subsection*{Acknowledgment}
The authors would like to thank the anonymous referee for the careful reading and valuable comments.
The first author was supported by NSFC 11671353 and NSFC 11971428.
The second author was
partially supported by Grant-in-Aid for Science Research
JSPS  (No.19H01795 and No.16H06339).

%\bibliographystyle{abbrv}
%\bibliography{bib.bib}

\end{document}